\documentclass[12pt,a4paper]{amsart}
\usepackage{preamble}

\title[Fast adjoint differentiation of chaos] 
{
Fast adjoint differentiation of chaos via computing unstable perturbations of transfer operators 
}

\begin{document}

\begin{abstract}

We devise the fast adjoint response algorithm for the gradient of physical measures (long-time-average statistics) of discrete-time hyperbolic chaos with respect to many system parameters. Its cost is independent of the number of parameters. 
The algorithm transforms our new theoretical tools, the adjoint shadowing lemma and the equivariant divergence formula  \cite{Ni_asl,TrsfOprt}, into the form of progressively computing $u$ many bounded vectors on one orbit. Here $u$ is the unstable dimension.
We demonstrate our algorithm on an example difficult for previous methods, a system with random noise, and a system of a discontinuous map.
We also give a short formal proof of the equivariant divergence formula.

Compared to the better-known finite-element method, our algorithm is not cursed by dimensionality of the phase space (typical real-life systems have very high dimensions), since it samples by one orbit.
Compared to the ensemble/stochastic method, our algorithm is not cursed by the butterfly effect, since the recursive relations in our algorithm is bounded.


\end{abstract}

\maketitle

\section{Introduction}

\subsection{Literature review}
\hfill
\vspace{0.1in}

The long-time-average statistic of a chaotic system is given by the physical measure, or the SRB measure, which is a model for fractal limiting invariant measures of chaos \cite{young2002srb,srbmap,srbflow}.
We are interested in the derivative of the long-time-average of an objective function with respect to the parameters of the system.
This derivative is also called the linear response or sensitivity, and it is one of the most basic numerical tools for many purposes.
For example, in aerospace design, we want to know the how small changes in the geometry would affect the average lift of an aircraft.
In climate change, we want to ask what is the temperature change caused by a small change of $CO_2$ level.
In machine learning, we want to extend the conventional backpropagation method to cases with gradient explosion; current practices have to avoid gradient explosion but that is not always achievable \cite{clip_gradients2,resnet}.
These are all linear responses, and an efficient algorithm would help the advance of the above fields.

There are several different formulas for the linear response; the most well-known are the ensemble formula and the derivative operator formula, and they are formally equivalent under integration-by-parts \cite{Gallavotti1996,Ruelle_diff_maps,Ruelle_diff_maps_erratum,Ruelle_diff_flow,Dolgopyat2004,Gouezel2006,Baladi2007}.
The ensemble formula is the average of the perturbation of individual orbits, which grows exponentially fast, requiring vast number of samples to take the average \cite{Lea2000,eyink2004ruelle,lucarini_linear_response_climate}.
The operator formula involves the derivative of the transfer operator, which is typically not defined pointwise.
The current operator methods typically approximates SRB measures by finite-elements: this is very inefficient for real-life systems which typically have very high dimensions \cite{Bahsoun2018,Galatolo2014,Antown2021a,Wormell2019a,TrsfOprt}.
A promising direction is to `blend' two formulas, and use the operator formula only for the unstable perturbations: this was attempted by Abramov and Majda, but they used a degenerate finite-element method for the unstable, so the cost is still cursed by dimensionality \cite{Abramov2008}.

To summarize, previous numerical algorithms are not suitable to be computed on an orbit, which is the most efficient way to sample a physical measure in high-dimensional phase spaces.
Our recent series of works devise new linear response formulas that are bounded on an orbit.
In this paper, we further transform the adjoint formula to one in the form of progressively computing $2u$ many recursive relations on an orbit, where $u$ is the unstable dimension.
Recursive adjoint linear response formulas were not provided by previous works, not to mention such a small number of recursions.

We decompose the linear response into shadowing and unstable contributions.
The shadowing contribution is given by the shadowing vector, which is the difference between two orbits close to each other but with perturbed parameters \cite{Bowen_shadowing,Pilyugin_shadow_linear_formula}.
For the purpose of developing an adjoint algorithm, we use our adjoint shadowing lemma, so that the main operator is moved away from the perturbation of systems \cite{Ni_asl,Ni_adjoint_shadowing}. 
The nonintrusive shadowing algorithm, which uses only the backward-running adjoint solutions, has been applied to computational fluid problems with over $3\times 10^6$ dimensions; its cost is independent of the number of parameters of the system \cite{Ni_nilsas,Ni_NILSS_AIAA_2016,Ni_NILSS_JCP,Ni_fdNILSS,Ni_CLV_cylinder,Blonigan_2017_adjoint_NILSS,Ni_NILSS_AIAA_2016}.
Shadowing does not give the full linear response; but when the ratio of unstable directions is low, shadowing can be a good approximation \cite{Ruesha}.

For the unstable contribution, recently, we derived a fast formula and the (tangent) fast response algorithm in the form of renormalized second-order tangent equations \cite{fr}.
The first issue of our previous work is that no intermediate quantity has a concrete physical meaning.
The second issue is that the cost is linear to the number of parameters, which can be very large for many real-life applications.
To address these issues, we recently developed the adjoint theory for the unstable contribution.
We showed that unstable contribution is given by the unstable perturbation of transfer operators, then we give a formula by equivariant divergences \cite{TrsfOprt}.
However, it is not yet clear that the equivariant divergence formula can be computed efficiently on an orbit with the fewest number of recursive relations.
This is the main goal of our current paper.

\subsection{Main results}
\hfill\vspace{0.1in}

An efficient and precise adjoint algorithm for chaos is an important open problem, and this paper gives an answer for an important basic scenario, the discrete-time systems with hyperbolicity.
The core of this paper is computing the equivariant divergence formula.
Since we are sampling the linear response on one orbit, in higher dimensions, our algorithm is much faster than previous adjoint methods based on integrations over the phase space.

We first give a short but only formal proof of the equivariant divergence formula for $\delta^u \tT^u\sigma$, the unstable perturbation of transfer operators on the unstable manifold.
\[ \begin{split}
  - \frac{\delta^u \tT^u\sigma} \sigma
  = \frac{\cL_{X^u} \sigma} \sigma
  = \cS(\div^v f_*) X + \div^v X.
\end{split} \]
Here $\cL$ is the Lie derivative of measures; $\div^v$ is the contraction by the unit unstable cube and its co-cube; $\cS$ is the adjoint shadowing operator (section~\ref{s:as} and section~\ref{s:UC}).
The main idea for this proof is to use the relation between Lie and Riemannian derivative.
Two other proofs, written more carefully, are given in \cite{TrsfOprt}, one is longer but more intuitive, the other is based on our forward tangent formula.

Then we derive the fast adjoint response algorithm, which computes the formula with the least possible cost.
The main difficulty is that the equivariant divergences are defined by unit $u$-vectors and $u$-covectors.
Renormalizing a $u$-vector has cost $O(Mu^2)$, which is the leading cost.
To resolve this, we derive a version of the formula that only needs renormalization occasionally.
With this, we can break a long orbit into small segments, and renormalize at interfaces.
We shall also write all our theories into matrix notations for the convenience of computations.

Since our adjoint formulas separate major computations away from $X$, the cost of a new $X$ is negligible.
In fact, the cost for the unstable contribution is further independent of the number of objectives.
Our algorithm samples the linear response by $2u+2$ recursive relations on a orbit, which can be viewed as a generalization of the MCMC method to derivatives of physical measures.
On our numerical examples in a 21-dimensional phase space, the cost of fast adjoint response is similar to finite-differences, but is almost unaffected by the number of parameters: 
this is orders of magnitude more efficient than previous precise adjoint algorithms.
We also study the convergence of the algorithm.

We also test the fast adjoint response algorithm on cases beyond the assumptions we proved the formulas.
Our algorithm works in the presence of stochastic forcing; it also works when the system is moderately discontinuous.
It seems to have robustness when there are no very strict center directions or when the center direction is not excited.
We treat flow cases in \cite{Ni_asl} and an incoming paper.

This paper is organized as follows.
Section~\ref{s:prelim} reviews the linear response theory and the adjoint shadowing lemma.
Section \ref{s:UC} gives a short formal proof of the equivariant divergence formula for the unstable contribution.
Section \ref{s:algorithm} develops the fast adjoint response algorithm.
Section \ref{s:examples} illustrates the algorithm on several numerical examples.

\section{Preliminaries}
\label{s:prelim}

\subsection{Hyperbolicity and linear response}
\hfill\vspace{0.1in}

Let $f$ be a smooth invertible map on a smooth Riemannian manifold $\cM$ of dimension $M$.
Assume that $K$ is a hyperbolic compact invariant set, that is, its tangent space has a continuous $f$-equivariant splitting into stable and unstable subspaces, $T_KM = V^s \bigoplus V^u$,
such that there are constants $C>0$, $0<\lambda < 1$, and
\[
  \max_{x\in K}|f_* ^{-n}|V^u(x)| ,
  |f_* ^{n}|V^s(x)| \le C\lambda ^{n} \quad \textnormal{for  } n\ge 0.
\]
Here $f_*$ is the pushforward operator on vectors,
when $\cM=\R^M$, $f_*$ is represented by multiplying with matrix $\partial f/\partial x$ on the left.
Define the oblique projection operators $P^u$ and $P^s$, such that
\[ \begin{split}
  v = P^u v + P^s v, \quad \textnormal{where} \quad P^u v\in V^u, P^s v\in V^s.
\end{split} \]
We further assume that $K$ is an attractor, that is, there is an open neighborhood $U$, called the basin of the attractor, such that $\cap_{n\ge0} f^nU=K$.

With some more assumptions, the attractor admits SRB measures, denoted by $\rho$.
In this paper, we define SRB measures as physical measures, that is, for every continuous observable $\Phi:\cM \rightarrow \R$ and almost all $x\in U$
\[ \begin{split}
  \lim_{N\rightarrow \infty} \frac 1N \sum _{k=0} ^{N-1} \Phi (f^kx) 
  = \rho(\Phi).
\end{split} \]
In other words, SRB measures give the long-time-average statistic of the chaotic system \cite{young2002srb,Ruelle_diff_maps}.

Assume that the system is parameterized by some parameter, $\gamma$; the linear response formula is an expression of $\delta \rho$ by $\delta f$, where
\[ \begin{split}
  \delta(\cdot):=\partial (\cdot)/\partial \gamma
\end{split} \]
may as well be regarded as small perturbations.
Linear response formulas are proved to give the correct derivative for various hyperbolic systems \cite{Ruelle_diff_maps,Dolgopyat2004,Jiang06}, yet it fails for some cases \cite{Baladi2007,Wormell2019,Gottwald2016}.
For an axiom A attractor $K$, the derivative of the SRB measure is given by:
  \[
    \delta \rho (\Phi)
    = \lim_{W\rightarrow\infty} \rho\left( \sum_{n=0}^W X(\Phi_n) \right) ,
  \]
where $X:=\delta f\circ f^{-1}$, $X(\cdot)$ is to differentiate in the direction of $X$, and $\Phi\in C^2$ is a fixed observable function.
Here $\Phi_n = \Phi\circ f^n$.

We call this the ensemble formula for the linear response.
There are several formulas expressing the same derivative in different ways.
Note that the integrand in this formula grows exponentially fast to $W$, and directly evaluating this integrand via ensemble or stochastic approaches is extremely expensive.
For faster computation, we decompose the linear response formula into two parts,
\[ \begin{split} \label{e:ruelle22}
  \delta \rho (\Phi)  = SC - UC ,
\end{split} \]
which we call the shadowing and the unstable contribution.
Like a Leibniz rule, the shadowing contribution accounts for the change of the location of the attractor,
while the unstable contribution accounts for the change of the measure if the attractor is fixed.

\subsection{Adjoint shadowing lemma and algorithm}
\hfill\vspace{0.1in}
\label{s:as}

We review our shadowing lemma and nonintrusive shadowing algorithm for the shadowing contribution.
First, the system of pulling-back covectors is also hyperbolic.
The pullback operator $f^*$ is the `adjoint', or basically the transpose, of the Jacobian $f_*$. That is, for any $x\in\cM$, $w\in T_{x}\cM$, and covector $\eta \in T^*_{fx}\cM$, 
\[ \begin{split}
  \eta(f_{*} w) =  f^* \eta (w).
\end{split} \]
On $\R^M$, covectors are inner-products with column vectors; however, they are in fact linear functions on vectors.
Covectors do not `point' to a direction; rather, they are more like gradients of functions.
Moreover, define the norm on covectors
\[ \begin{split}
  |\eta| = \sup_w \eta(w)/ |w|.
\end{split} \]
Let $\cX^{*\alpha}$ denote the space of Holder covector fields.

Define the adjoint projection operator, $\cP^u$ and $\cP^s$, as the adjoint of $P^u$ and $P^s$; in $\R^M$ these are just transpose matrices.
Hence, for any $w\in T_{x}\cM$, $\eta \in T^*_{x}\cM$,
\[ \begin{split}
  \eta (P^u w) = \cP^u \eta(w),
  \quad \textnormal{} \quad 
  \eta (P^s w) = \cP^s \eta(w).
\end{split} \]
Define the image space of $\cP^u$ and $\cP^s$ as $V^{u*}$ and $V^{s*}$; we can show that they are in fact unstable and stable subspaces for the adjoint system.
Hence, if we pullback $u$-many covectors for many steps, they automatically occupy the unstable subspace, since their unstable parts grow while stable parts decay.

\begin{theorem} [adjoint shadowing lemma for discrete-time \cite{Ni_asl}] \label{t:AS}
  The adjoint shadowing operator $\cS$ is equivalently defined by the following characterizations:
  \begin{enumerate}
  \item 
  $\cS$ is the linear operator $\cS:\cX^{*\alpha}(K) \rightarrow \cX^{*\alpha} (K)$, such that
  \[ \begin{split}
    \rho(\om S(X)) = \rho(\cS(\om) X)
    \quad \textnormal{for any } X\in\cX^\alpha(K).
  \end{split} \]
  Hence, let $\nu=\cS(d\Phi)$, then the shadowing contribution is
  \[ \begin{split}
    SC = \rho(d\Phi S(X)) = \rho(\cS(d\Phi) X)
    =  \lim_{N\rightarrow \infty} \frac 1{N} \sum_{n=1}^N  \nu_n X_n \,.
  \end{split} \]
  \item 
  $\cS(\om)$ has the expansion formula given by a `split-propagate' scheme,
  \begin{equation*}
    \cS ( \om ) := \sum_{n\ge 0} f^{*n} \cP^s \om_n
    -  \sum_{n\le -1}  f^{*n} \cP^u \om_n\,.
  \end{equation*}
  \item 
  The shadowing covector $\nu=\cS(\om)$ is the unique bounded solution of the inhomogeneous adjoint equation,
  \end{enumerate}
  \[ \begin{split}
  \nu = f^* \nu_1 + \om, 
  \quad \textnormal{where} \quad \nu_1:=\nu \circ f.
  \end{split} \]
\end{theorem}

Characterization (c) can be efficiently and conveniently recovered by the nonintrusive adjoint shadowing algorithm \cite{Ni_nilsas}.
This is an `adjoint' algorithm, which means that its cost is independent of the number of parameters.
More specifically, we solve the nonintrusive adjoint shadowing problem,
\[ \begin{split} 
  \nu =  \nu ' + \ueps a \,,
  \quad  \mbox{s.t. }
  \ip{\nu_0, \ueps_0}=0.
\end{split}\]
where $\nu'$ is the inhomogeneous adjoint solution with $\nu'_{N}=0$, $\ueps$ is a matrix with $u$ columns of homogeneous adjoint solutions, and the orthogonality is at step 0.
Nonintrusive means to solve $O(u)$ many solutions of the most basic recursive relations.
The nonintrusive adjoint shadowing algorithm was used on fluid problems with $u=8$, and $M\approx3\times10^6$; the cost was on the same order of simulating the flow, and the approximate derivative was useful \cite{Ni_nilsas,Ni_NILSS_JCP,Ni_fdNILSS,Ni_CLV_cylinder}.

Intuitively, the boundedness property is achieved by removing $u$-many unstable modes, since unstable modes are the main cause of unboundedness.
We have some liberty in choosing how to remove the unstable modes.
Here we choose to orthogonally project out the unstable modes at the last step, where the unstable modes are the most significant: this one is the easiest to implement.
It is also possible to do the oblique unstable projection using both the tangent and adjoint unstable modes.
We may also use least-squares minimization over the entire orbit, but solving the minimization is more complicated, and appendix~\ref{a:schur} shows how to do this with cost below the other parts of the fast response algorithm. 
It is to be tested which mimic is the best, both when the nonintrusive shadowing is used stand-alone and as a part of the fast response algorithms.
For the examples in this paper we found no obvious differences in the results.

When $u\ll M$, if the system has fast decay of correlation, and $X$ and $d\Phi$ are not particularly aligned with unstable directions, the shadowing contribution can be a good approximation of the entire linear response \cite{Ruesha}.
However, when the unstable contribution is large, or when better accuracy is desired, for example near design optimal, we need to further compute the unstable contribution \cite{RepolhoCagliari2021,Lasagna2019,BloniganPhdThesis}.
Since shadowing is also an important part for the unstable contribution, no work will be wasted.

\section{Unstable contribution}
\label{s:UC}

\subsection{Notations}
\hfill\vspace{0.1in}

The unstable contribution is defined as
\[ \begin{split}
  UC := \lim_{W\rightarrow\infty} UC^W, \quad
  UC^W := - \rho \left(\sum_{n=-W}^W  X^u (\Phi _n) \right). 
\end{split} \]
The integrand in the unstable contribution grows exponentially fast.
To treat this, integrate-by-parts on the unstable manifold, we get
\begin{equation} \begin{split} \label{e:uc}
  UC^W = \rho \left(\psi  \diverg_\sigma^u  X^u\right),
  \quad \textnormal{where} \quad 
  \psi := \sum_{m=-W}^W (\Phi \circ f^m-\rho(\Phi).
  \end{split} \end{equation}
Here $\diverg_\sigma^u$ is the submanifold divergence on the unstable manifold under the conditional SRB measure.
The unstable divergence turns out to be the derivative of the transfer operator on unstable manifolds \cite{TrsfOprt}; the norm of the integrand now is $O(\sqrt W)$, much smaller than the ensemble formula.
However, the directional derivatives of $X^u$ are distributions, so we need better formulas for computations.

Let $\nabla$ be the Riemannian derivative, $\cL$ be the Lie derivative.
Let $\tilde e:=e_1\wcw e_u$ be the unit $u$-dimensional cube spanned by unstable vectors, $\ip{\cdot, \cdot}$ is the inner product between $u$-vectors;
Let $\nabla f_*$ be the symmetric Hessian tensor defined by the Leibniz rule, 
\[ \begin{split}
  (\nabla_{\tv} f_*) \te 
  := \nabla_{f_* \tv } (f_*\te) - f_*\nabla_{ \tv} \te
  = \sum_{i=1}^u f_*e_1\wcw (\nabla_{\tv} f_*) e_i \wcw f_*e_u 
  \\
  = \sum_{i=1}^u f_*e_1\wcw (\nabla_{e_i} f_*) \tv \wcw f_* e_u
  =: (\nabla_{\te} f_*) \tv .
\end{split} \]
Let $\teps$ be the unit unstable $u$-covectors such that $\teps(\te)=1$, more specifically,
\[ \begin{split}
  \eps := \eps^1 \wcw \eps^u, \quad 
  \teps := \eps/\eps(\te), 
  \quad \textnormal{ where} \quad 
  \eps^i\in V^{u*}.
\end{split} \]

There are two different divergences on an unstable manifold, the first is taken within the unstable submanifold, or $u$-divergence, which applies to vector fields parallel to the submanifold,
\[ \begin{split}
  \div^u X^u:= \ip{\nabla_\te X, \te}.
\end{split} \]
The second kind is the equivariant unstable divergence, or $v$-divergence, as
\[ \begin{split}
  \div^v X:= \teps \nabla_\te X.
\end{split} \]
Note this divergence applies to any vector fields.
We define the $v$-divergence of the Jacobian matrix $f_*$, which is a Holder continuous covector field on the attractor.
\[ \begin{split}
  \div^v f_* :=\frac{\teps_1 \nabla_{\te} f_* } {|f_* \te|},
  \quad \textnormal{} \quad 
  (\div^v f_*) X :=\frac{\teps_1 (\nabla_{\te} f_*)X } {|f_* \te|},
  \quad \textnormal{where} \quad 
  \eps_1(x) := \eps(fx).
\end{split} \]

\subsection{A short formal proof of the equivariant divergence formula}
\hfill\vspace{0.1in}

This subsection gives a short and formal proof of the equivariant divergence formula for $UC$ via Lie derivatives.
Let $\sigma$ be the conditional SRB measure on unstable manifold $\cV^u$.
Let $\tT^u$ be the transfer operator on unstable manifolds caused by the flow of $X^u$, which depends on $\gamma$.
After the perturbation, the new measure on $\cV^u(x)$ has the following expressions.

\begin{theorem} [equivariant divergence formula \cite{TrsfOprt}] \label{t:fast adj formula}
  \[ \begin{split}
  \div ^u_\sigma X^u  
  = - \frac{\delta^u \tT^u\sigma} \sigma
  = \frac{\cL_{X^u} \sigma} \sigma
  = (\cS\omega) X + \div^v X,
  \quad \textnormal{where} \quad 
  \omega:=\div^v f_*.
  \end{split} \]
  Here $\cS$ is the adjoint shadowing operator.
\end{theorem}

\begin{remark*}
(1)
The equivalence between the first three expressions are just due to different interpretations of $UC$.
(2)
The last expression is `adjoint' in the utility sense, that is, to compute $\div ^u_\sigma X^u$ for a new $X$, we only need to apply $\cS \omega$ on $X$ and contract the 2-tensor $\nabla X$, whose cost are marginal compared to computing quantities not affected by $X$, such as $\te, \teps, \cS \omega$.
(3)
The derivatives in the equivariant formula hits only $f_*$ and $X$, which are smooth. Hence, all quantities are well-defined pointwise.
(4)
Compared to previous linear response formulas, in particular those well-defined pointwise \cite{Ruelle_diff_maps_erratum,Gouezel2008}, this formula is given by recursive relations on an orbit, and we only require $2u+2$ such recursions, which should be optimal.
It also has an explicit physical interpretation of unstable transfer operators, which is convenient for deriving explicit formulas for transient perturbations or continuous time cases.
\end{remark*}

\begin{proof}
This formal proof omits a lot of detailed estimations which coincide with \cite{TrsfOprt}, and mainly presents the different part, that is, here we use the relation between Lie and Riemannian derivative to decompose the unstable perturbation.

Denote the conditional measure by the $u$-covector $\eps$, 
\[ \begin{split}
  \eps = \sigma \teps = \lim_{T\rightarrow \infty} f^{*-T} \teps_{-T},
  \quad \textnormal{} \quad 
  e = \lim_{T\rightarrow \infty} f_*^{T} \te_{-T}.
\end{split} \]
So $\eps e=1$.
The re-distribution of the conditional measure caused by $X^{u}$ is
\begin{equation} \begin{split} \label{e:xiaoyue}
  \frac{\delta \cT\sigma} \sigma
  = \frac{(\cL_{-X^{u}} \eps) e}{\eps e}
  = \frac{\eps (\cL_{X^{u}} e) }{\eps e}
  = \frac 1{\eps e} \left( \eps \nabla_{X^{u}}e - \eps \nabla_e X^u   \right)
  = \eps  \left(\nabla_{X^{u}}e - \nabla_e X +\nabla_e X^s  \right).
\end{split} \end{equation}
For the second term, by definition of $\div^v$,
\[ \begin{split}
   \eps  \nabla_e X =: \div^v X.
\end{split} \]

For the first term in \eqref{e:xiaoyue}, use the Leibniz rule,
\[ \begin{split}
  \eps \nabla_{X^{u}} e
  = \eps \nabla_{f_*^T f_*^{-T}X^{u}} f_*^T e_{-T}
  = \eps f_*^T \nabla_{f_*^{-T}X^{u}} e_{-T}
  + \eps \sum_{n=0}^{T-1} f_*^{T-n-1}(\nabla_{f_*^{n-T} X^u} f_*) f_*^{n} e_{-T}
  \\
  = \eps_{-T} \nabla_{f_*^{-T} X^u } e_{-T}
  + \sum_{n=0}^{T-1} \eps_{n+1-T} (\nabla_{f_*^{n-T} P^u X} f_*) e_{n-T}
  = X \sum_{n=0}^{T-1} f^{*n-T} \cP^u (\div^v f_*)_{n-T}
\end{split} \]
Here $\eps_{-T} \nabla_{f_*^{-T}X^{u}} e_{-T}\rightarrow 0$ since $e_{-T}$ and $\eps_{-T}$ are unit fields.

\begin{figure}[ht] \centering
  \includegraphics[width=0.8\textwidth]{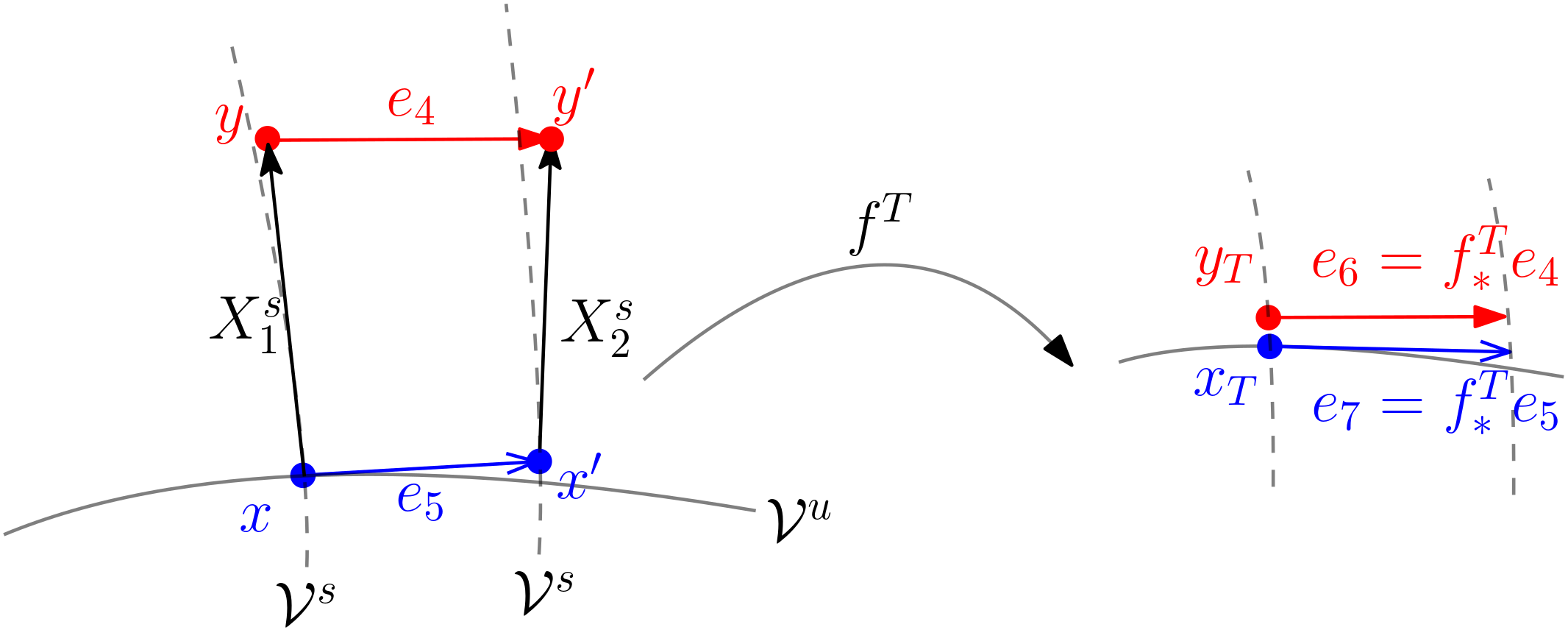}
  \caption{Intuitions for $\nabla_eX^s$. $x'-x=e_5=e(x)$, $y-x=X^s_1:=X^s(x)$,  $y'-x'=X^s_2:=X^s(x')$, and $e_4:=y'-y$. 
  Note that $e_4$ is pushforward of $e$ along $X^s$.}
  \label{f:bala}
\end{figure}

Figure~\ref{f:bala} gives an intuitive pictorial explanation for $\nabla_eX^s$ in the last term in~\eqref{e:xiaoyue}.
Intuitively, assuming $u=1$ and $\cM=\R^M$,
\[ \begin{split}
  \nabla_eX^s
  = X^s_2 - X^s_1 
  = (y'-x') - (y-x)
  = (y'-y) - (x'-x)
  = e_4-e_5
\end{split} \]
More specifically, we denote the flow of $X^s$ by $\xi$, and $\xi_*e$ be the pushfoward of $e$ by $\xi^\tau$ for a small interval of $\tau$.
Then $\cL_{X^s}\xi_*e = 0$, so
\[ \begin{split}
  \nabla_{e} X^s 
  = \nabla_{\xi_*e} X^s 
  = \nabla_{X^s}{\xi_*e}.
\end{split} \]
Apply the Leibniz rule, also note that $\xi_*$ is the identity when not differentiated, 
\[ \begin{split}
  \eps \nabla_{X^s} \xi_*e
  = \eps_T f_*^T \nabla_{X^s} \xi_*e
  = \eps_T \nabla_{f_*^T X^s} f_*^T \xi_*e
  - \eps_T \sum_{n=0}^{T-1} f_*^{T-n-1}(\nabla_{f_*^n X^s} f_*) f_*^{n} e
\end{split} \]
Intuitively, because $f_*^T \xi_*e \rightarrow f_*^Te $, in figure~\ref{f:bala}, $\eps_T(e_6-e_7)\rightarrow0$; hence, $ \eps_T \nabla_{f_*^T X^s} f_*^T \xi_*e \rightarrow 0$. 
We do not prove this intuition, which is closely related to a lemma of absolute continuity.
Acknowledging this, we have
\[ \begin{split}
  \eps \nabla_{e} X^s
  = \eps \nabla_{X^s} \xi^*e
  = - \sum_{n=0}^{T-1} \eps^{n+1}(\nabla_{f_*^n P^s X} f_*) e_n
  = - X \sum_{n=0}^{T-1} f^{*n} \cP^s (\div^vf_*)_n.
\end{split} \]

Summarizing,
\[ \begin{split}
  \frac{\delta \cT\sigma} \sigma
  = -\div^v X
  - X \sum_{n\ge0} f^{*n} \cP^s (\div^vf_*)_n
  + X \sum_{n\le-1} f^{*n} \cP^u (\div^vf_*)_n.
\end{split} \]
By (b) of the adjoint shadowing lemma,
\[ \begin{split}
  \frac{\delta \cT\sigma} \sigma
  = -\div^v X - \cS(\div^v f_*).
\end{split} \]
\end{proof}

We previously developed a tangent formula and algorithm which computes the unstable contribution via only forward iterations \cite{fr}; the equivalence between the tangent and our current adjoint formula is given in \cite{TrsfOprt}.
Comparing to the adjoint algorithm to be given in this paper, the (tangent) fast response algorithm runs only forwardly, so it is faster than the adjoint algorithm, though the number of flops (float point operations) are the same for one parameter.
But the tangent algorithm's cost is linear to the number of parameters.
There are other similar tangent algorithms later on, but are less efficient \cite{Chandramoorthy2021a}.
Another issue of the tangent formulas is that they seem mysterious due to lack of intermediate quantities with physical meanings.

\section{Fast adjoint response algorithm}
\label{s:algorithm}

This section develops the algorithm.
We show how to compute the equivariant divergence formula on a long orbit divided into small segments: this treatment finally realizes our claim that we only need to track $2u$ recursive relations.
Moreover, this multi-segment treatment will help reduce numerical error and the frequency of renormalizations, and get rid of computing determinants.
We shall also write everything in matrix notations.
Then we give a detailed procedure list of the algorithm.

\subsection{Notations}
\hfill\vspace{0.1in}

Our convention for subscripts on multi-segments is shown in figure~\ref{f:subscript}.
We divide an orbit into small segments of $N$ steps.
The $\alpha$-th segment consists of step $\alpha N$ to $\alpha N+N$,
where $\alpha$ runs from $0$ to $A-1$;
notice that the last step of segment $\alpha$ is also the first step of segment $\alpha+1$.
We use double subscripts, such as $x_{\alpha, n}$,
to indicate the $n$-th step in the $\alpha$-th segment,
which is the $(\alpha N+n)$-th step in total.
Note that for some quantities defined on each step, for example,
$e_{\alpha, N}\neq e_{\alpha+1, 0}$,
since renormalization is performed at the interface across segments.
Continuity across interfaces is true only for some quantities, 
such as shadowing covectors $\nu$, $\tnu$,
and unit unstable cubes $\te$, $\teps$.
Later, we will define some quantities on the $\alpha$-th segment,
such as $C_\alpha, d_\alpha$: their subscripts are the same as the segment they are defined on.
For quantities to be defined at interfaces, 
such as $Q_\alpha, R_\alpha, b_\alpha$,
their subscripts are the same as the total step number of the interface divided by $N$.

\begin{figure}[ht] \centering
  \includegraphics[width=0.7\textwidth]{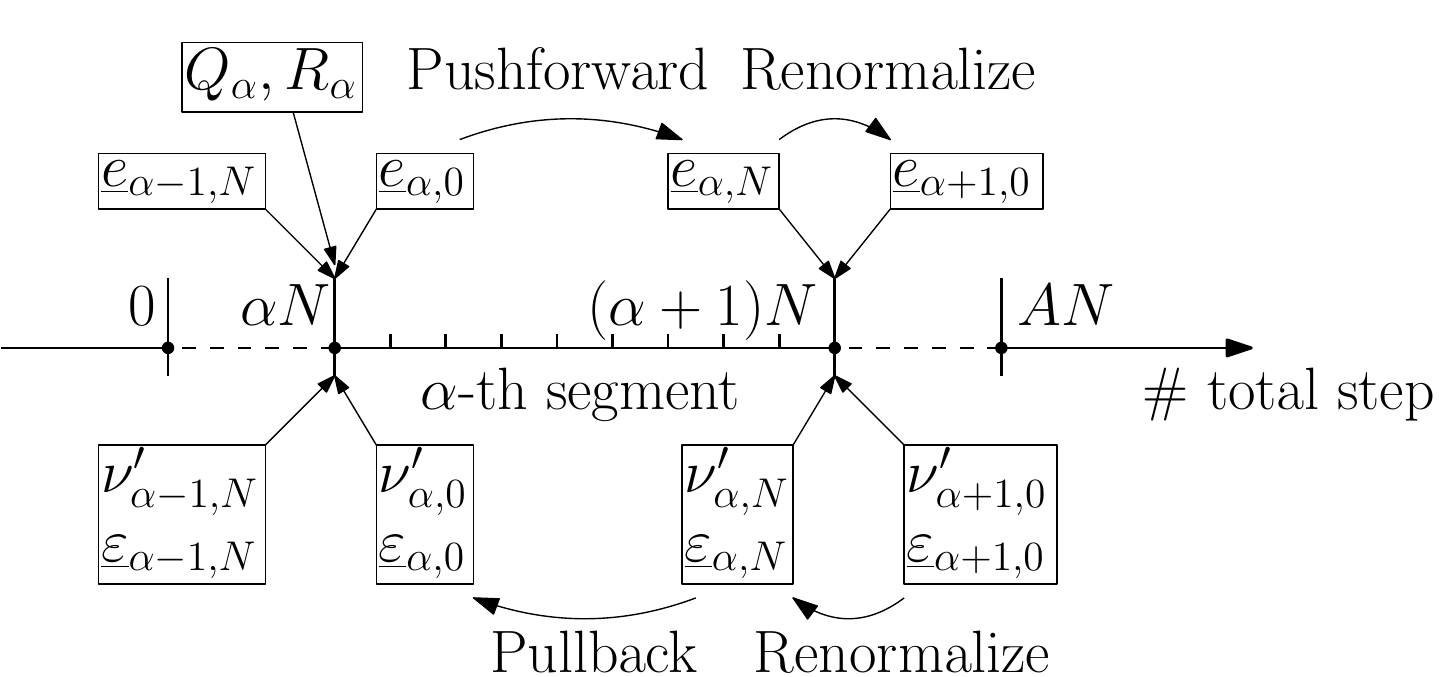}
  \caption{Subscript convention on multiple segments.}
  \label{f:subscript}
\end{figure}

We introduce some matrix notations. Let $I$ be the $u\times u$ identity matrix.
Moreover, denote $\underline e:=[e_1,\cdots,e_u]$, $\underline \eps:=[\eps^1,\cdots,\eps^u]$,
\[ \begin{split}
  \underline \eps^T \underline e := [\eps^i e_j]_{ij},
  \quad \textnormal{so} \quad 
  \eps e = \det(\underline \eps^T \underline e).
\end{split} \]
Here $[\cdot]_{ij}$ is the matrix with $(i,j)$-th entry given inside the bracket.
Similarly, 
\[ \begin{split}
  \ueps^T \ueps := \left[\ip{\eps^i, \eps^j}\right]_{ij}, \quad
  \ue^T \ue := \left[\ip{e_i,e_j}\right]_{ij}.
\end{split} \]
where $\ip{,}$ is the inner product.
For any vector $w$,
\[ \begin{split}
  \eps_1 (\nabla_e f_*) w = \sum_{i=1}^u \det \mat{
  \eps_1^1 f_* e_1 &\cdots &  \eps_1^1 (\nabla_{e_i}f_*)w &\cdots &\eps_1^1 f_* e_u &\\
  \vdots && \vdots && \vdots \\ 
  \eps_1^u f_* e_1 &\cdots &  \eps_1^u (\nabla_{e_i}f_*)w &\cdots &\eps_1^u f_* e_u &\\
  }.
\end{split} \]
The expression for $\nabla_e X$ is similar.

\subsection{Multi-segment treatment}
\hfill\vspace{0.1in}

When $\eps$ and $e$ have duality, the expressions in the last subsection admits significant reductions, relieving us from computing the determinant.
This section shows how to achieve such duality, while keeping tangent and adjoint solutions bounded via renormalizations.
Also for cost reasons, we want to perform renormalizations only occasionally.
The solution is to divide a long orbit into multiple segments, renormalize at interfaces, and use only non-normalized quantities within each segment.

\begin{lemma} [formulas using non-normalized $e$ and $\eps$] \label{l:violin}
For any unstable $e$ and $\eps$,
\[ \begin{split}
  \teps \nabla_{\te} X = \frac{1} {\eps e} \eps \nabla_{e} X ;
  \quad \quad 
  \omega = \frac{1} {\eps_1 f_* e} \eps_1 \nabla_e f_* .
\end{split} \]
\end{lemma}

\begin{proof}
Recall that $\teps = \eps/\eps(\te)$, $\te = e/|e|$, so
\[ \begin{split}
  \teps \nabla_{\te} X
  = \frac{\eps} {\eps \te} \nabla_{ \frac e {|e|} } X
  = \frac{\eps} {\eps (|e| \te)} \nabla_{ e } X
  = \frac{1} {\eps e} \eps \nabla_{e} X.
\end{split} \]
\[ \begin{split}
  \omega
  := \teps_{1} \frac{\nabla_{\te} f_* } {|f_*\te|} 
  = \frac{\eps_1}{\eps_1 \te_1} \frac{\nabla_{e/|e|} f_* } {|f_*\te|} 
  = \frac{\eps_1}{\eps_1 \te_1} \frac{\nabla_e f_* } {|e| |f_*\te|} 
  = \frac{1}{\eps_1 f_*e} \eps_{1} \nabla_e f_* ,
\end{split} \]
since $|e| |f_*\te|\te_1 = |f_* e|\te_1 =f_* e $.
\end{proof}

\begin{lemma}[duality condition] \label{l:keep dual}
  Within any segment $\alpha$, let
  \[ \begin{split}
  e_{\alpha, n}:= f_* e_{\alpha, n-1}, \quad 
  \eps_{\alpha, n}:= f^* \eps_{\alpha, n+1},
  \end{split} \]
  For any invertible matrix $R_{\al}$, when moving across segments, if
  \[ \begin{split}
  (\ueps^T \ue)_{\al,0} = I  , \quad 
  \ue_{\al-1, N} = \ue_{\al,0} R_\al, \quad 
  \ueps_{\al-1, N}  = \ueps_{\al, 0} R^{-T}_\al,
  \end{split} \]
  Then for all suitable $n$ in segment $\al-1$, we have
  $(\underline\eps^T \underline e)_{ n} = (\underline\eps_1^T f_*\underline e)_{ n} = I $, and
  \[ \begin{split}
  (\teps \nabla_\te X)_n = \sum_{i=1}^u (\eps^i \nabla_{e_i} X)_n , \quad 
  \omega_{ n} = \sum_{i=1}^u \eps_{n+1}^i (\nabla_{e_{n,i}}f_*) .
  \end{split} \]
\end{lemma}

\begin{remark*}
(1)
Typically we perform QR factorization to $\ue_{\al-1,N}$, 
so $\ue_{\al,0}$ is orthogonal, and $R_\al$ is upper triangular.
(2)
If not for round-off errors,
$\ueps_{\al,0}^T \ue_{\al-1, N}=\ueps_{\al,0}^T \ue_{\al,0} R_\al =R_\al$.
To guarantee the duality, we compute via the more expensive expression,
$\ueps_{\al-1, N} = \ueps_{\al, 0} (\ueps_{\al,0}^T \ue_{\al-1, N})^{-T}$.
\end{remark*}

\begin{proof}
  Since $\underline\eps_1^T f_*\underline e = \ueps_1^T \ue_1 = \ueps^T \ue$ is constant on segment $\al-1$,
  we only need to show $(\ueps^T\ue)_{\al-1, N}=I$, and this is because
  \[ \begin{split}
    (\ueps^T\ue)_{\al-1, N} 
    = (\ueps_{\al, 0} R_\al^{-T})^T \ue_{\al,0} R_\al
    = R_\al^{-1} \ueps_{\al, 0}^T \ue_{\al,0} R_\al
    = R_\al^{-1} R_\al = I.
  \end{split} \]
  Then use lemma~\ref{l:violin}.
\end{proof}

In the nonintrusive adjoint shadowing algorithm, 
we need to compute a particular inhomogeneous adjoint solution $\nu'$.
To avoid it from growing too large, we need to throw out its unstable part after every segment;
meanwhile, we should keep a continuous affine space, $\nu'+\spanof \ueps$, 
so that later we can recover a continuous shadowing covector $\nu$.
This was a standard part of nonintrusive shadowing algorithms,
but since we changed the renormalizing scheme on $\ueps$,
we shall re-derive the renormalizing scheme for $\nu'$ and continuity conditions.

\begin{lemma} [inhomogeneous continuity] \label{l:linus}
  Let 
  \[ \begin{split}
    \nu'_{\al-1, N} = \nu'_{\al, 0} - \ueps_{\al, 0} b_\al,
    \quad \textnormal{where} \quad 
    b_\al = \left( \ueps_{\al, 0}^T \ueps_{\al, 0} \right)^{-1}  \ueps_{\al, 0}^T \nu'_{\al, 0} ,
  \end{split} \]
  then $\nu'_{\al-1, N}\perp V^{u*}_{\al-1,N}$,
  and the affine space $\nu'_{\al-1, N}+ \spanof \ueps_{\al-1, N} = \nu'_{\al, 0}+ \spanof \ueps_{\al, 0} $.
  If the shadowing covector is represented nonintrusively as 
  \[ \begin{split}
  \nu_{\al,n} = \nu'_{\al,n} + \ueps_{\al, n} a_\al,    
  \end{split} \]
  then the continuity of $\nu$ across segments, $\nu_{\al,N} = \nu_{\al+1, 0}$, is equivalent to 
  \[ \begin{split}
    a_{\al-1} = R_\al^T (a_\al + b_\al).
  \end{split} \]
\end{lemma}

\begin{proof}
  To see the orthogonality, notice that $ V^{u*}_{\al-1,N}= \spanof \ueps_{\al-1, N} = \spanof \ueps_{\al, 0}$, so
  \[ \begin{split}
    \ueps_{\al, 0}^T \nu'_{\al-1, N} 
    = \ueps_{\al, 0}^T \nu'_{\al, 0} - \ueps_{\al, 0}^T \ueps_{\al, 0} b_\al = 0.
  \end{split} \]
  The continuity of the affine space follows from definitions.
  For $\nu$,
  \[ \begin{split}
    \nu_{\al-1,N} = \nu_{\al, 0}
    \iff
    \nu'_{\al-1,N} + \ueps_{\al-1, N} a_{\al-1} = \nu'_{\al,0} + \ueps_{\al, 0} a_{\al}\\
    \iff
    - \ueps_{\al, 0} b_\al  + \ueps_{\al-1, N} a_{\al-1} = \ueps_{\al, 0} a_{\al}\\
    \iff
    \ueps_{\al, 0} \left(R_\al^{-T} a_{\al-1} - a_{\al} -  b_\al \right)=0
    \iff
    a_{\al-1} =R_\al^T (a_{\al} +  b_\al) .
  \end{split} \]
  Here the last equivalence is because $\ueps$ is a basis.
\end{proof} 

\subsection{Procedure list}
\hfill\vspace{0.1in}
\label{s:procedure}

This subsection gives a detailed procedure list of the fast adjoint response algorithm.
When $\cM=\R^M$, corresponding simplifications are explained.
The subscript explanation is in figure~\ref{f:subscript}.

\begin{enumerate}[label={\roman*.}]

\item 
Evolve the dynamical system for a sufficient number of steps before $n=0$,
so that $x_{0,0}$ is on the attractor at the beginning of our algorithm. 
Then, evolve the system from segment $\alpha=0$ to $\alpha=A-1$,
each containing $N$ steps, to obtain the orbit,
\[ \begin{split}
  x_{\alpha, n+1} = f(x_{\alpha, n}), \quad x_{\alpha+1,0} = x_{\alpha,N}.
\end{split} \]

\item 
Compute tangent solutions.
Set random initial conditions for each column in $\underline e:=[e_1,\cdots,e_u]$.
For $\alpha$ from $0$ to $A-1$ do the following:
\begin{enumerate}
  \item
  From initial conditions, solve tangent equations, $\alpha$ neglected,
  \[ \begin{split}
    \underline e_{n+1} = f_* \underline e_{ n}.
  \end{split} \]
  Here $f_*$ is the pushforward operator.
  In $\R^M$, the vectors are column vectors, and $f_*$ is the Jacobian matrix,
  \[ \begin{split}
    f_* = [\partial f^i/\partial z^j]_{ij},
  \end{split} \]
  multiplied on the left of column vectors.
  Here $[\cdot]_{ij}$ is the matrix with $(i,j)$-th entry given inside the bracket,
  $f^i$ is the $i$-th component of $f$, $z^j$ is the $j$-th coordinate of $\R^M$.

  \item 
  At step $N$ of segment $\alpha$, 
  orthonormalize $\underline e$ with a QR factorization
  \[ \begin{split}
    \underline e_{\alpha, N} = Q_{\alpha+1} R_{\alpha+1}.
  \end{split} \]
  
  \item
  Set initial conditions of the next segment, 
  \[
    \underline e_{\alpha+1, 0} = Q_{\alpha +1}.
  \]
\end{enumerate}

\item 
Compute adjoint solutions.
Set terminal condition $\nu'_{A-1,N}=\tnu'_{A-1,N}=0$, 
and $\ueps_{A-1,N}$ such that $(\ueps^T\ue) _{A-1,N} = I$;
here $\ueps= [\eps^1,\cdots,\eps^u]$.
In $\R^M$ we can set $\ueps_{A-1,N} = Q_A R_A^{-T}$.
For $\alpha$ from $A-1$ to $0$ do the following:
\begin{enumerate}
  \item
  From terminal conditions, solve homogeneous adjoint equations,
  \[ \begin{split}
    \underline \eps_{n-1} = f^* \underline \eps_{ n}.
  \end{split} \]
  In $\R^M$, covectors are represented by inner-product with column vectors,
  and the pullback operator $f^*$ at $x_n$ is the transposed Jacobian matrix evaluated at $x_{n-1}$,
  multiplied on the left of column vectors.

  \item  \label{step:omega}
  Compute the inhomogeneous term $\omega$ for the modified adjoint shadowing equation,
  note that we do not compute $\omega_n$ at $n=N$:
  \[ \begin{split}
    \omega_{n-1} = \sum_{ i=1 }^u \eps_{n}^i (\nabla_{e_{n-1,i}}f_*) .
  \end{split} \]
  Here $\nabla_{(\cdot)}f_*$ is the Riemannian derivative of $f_*$ \cite{fr}. 
  In $\R^M$, 
  \[ \begin{split}
    \omega_{n-1} = \sum_{i=1}^u \sum_{l=1}^M \sum_{k=1}^M
    \left[\eps_{n,l}^i \pp{^2 f_{n-1}^l}{x^k\partial x^j} e_{n-1,i}^k  \right]_j
    \in\R^M,
  \end{split} \]
  where $\eps_l, e^k$ are components in $\R^M$.

  \item
  Solve inhomogeneous adjoint equations,
  \[ \begin{split}
    \nu'_{n-1} = f^* \nu'_{n} + d\Phi_{n-1}, \quad
    \tnu'_{n-1} = f^* \tnu'_{n} + \omega_{n-1}.
  \end{split} \]
  Note that the terminal values come from terminal conditions.

  \item 
  At step $0$ of segment $\alpha$, compute and store
  \[ \begin{split}
    b_\al = \left( \ueps_{\al, 0}^T \ueps_{\al, 0} \right)^{-1}  \ueps_{\al, 0}^T \nu'_{\al, 0} ,\quad
    \tilde b_\al = \left( \ueps_{\al, 0}^T \ueps_{\al, 0} \right)^{-1}  \ueps_{\al, 0}^T \tnu'_{\al, 0} .
  \end{split} \]
  Here $\underline \eps^T \underline \eps := [\ip{\eps_i, \eps_j}]_{ij}$,
  and $\ip{\cdot,\cdot}$ in $\R^M$ is the inner product between column vectors.
  
  \item
  Set initial conditions of the next segment, 
  \[ \begin{split}
    \ueps_{\al-1, N} = \ueps_{\al, 0} (\ueps_{\al,0}^T e_{\al-1, N})^{-T}, \quad
    \nu'_{\al-1, N} = \nu'_{\al, 0} - \ueps_{\al, 0} b_\al,\quad
    \tnu'_{\al-1, N} = \tnu'_{\al, 0} - \ueps_{\al, 0} \tilde b_\al .
  \end{split} \]
\end{enumerate}

\item \label{step:ni} 
Solve the nonintrusive adjoint shadowing problem for $\{a_\alpha\}_{\al=0}^{A-1}$:
\[\begin{split}
  a_0 = - b_0;
  \quad \textnormal{} \quad 
  a_{\al} =  R^{-T}_\al a_{\al-1} - b_\al,
  \quad \alpha=1,\ldots,A-1.
\end{split}\]
Solve the same problem again, with $b$ replaced by $\tilde b$, for $\tilde a$.
We may also solve a least-squares version of the nonintrusive shadowing, as detailed in appendix~\ref{a:schur}.
Compute shadowing covectors
\[
 \nu_\alpha  := \cS(d\Phi) = \nu'_\alpha + \underline \eps_\alpha a_\alpha, \quad
 \tnu_\alpha := \cS(\omega) = \tnu'_\alpha + \underline \eps_\alpha \tilde a_\alpha.
\]

\item
Compute the shadowing contribution,
\[ 
  SC
  = \lim_{A\rightarrow\infty}
    \frac 1 {AN} \sum_{\alpha=0}^{A-1} \sum_{n=1}^N
     \nu_{\alpha,n} X_{\alpha,n}.
\]
Here $X_n:=(\partial f/\partial \gamma)_{n-1}$, 
where $\gamma$ is the parameter of the system.
Note that we do not compute $X_n$ at $n=0$.
Also note that $X$ is used only after major computations have been done;
no previous procedure depends on $X$.

\item 
\begin{enumerate}
  \item  Compute the unstable divergence.
  \[ \begin{split}
  \div ^u_\sigma X^u_n  
  = - \frac{\delta^u \tT^u\sigma_{-1}} \sigma (x_n)
  = (\tnu X)_n + \left( \eps \nabla_{e} X \right)_n.
  \end{split} \]
  In $\R^M$, 
  \[ \begin{split}
  \left( \eps \nabla_{e} X \right)_n
  = \sum_{i=1}^u \sum_{j=1}^M \sum_{l=1}^M
  \eps^i_{n,l} \ppp{f^l_{n-1}}{\gamma}{x^j} e^j_{n-1,i}.
  \end{split} \]

  \item Compute the unstable contribution
  \[ \begin{split}
  UC^W 
  = \lim_{A\rightarrow\infty} \frac 1 {AN} \sum_{\alpha=0}^{A-1} \sum_{n=1}^N
     \psi_{\alpha,n} \div ^u_\sigma X^u_{\alpha,n},
  \end{split} \]
  where $\psi := \sum_{m=-W}^W (\Phi \circ f^m-\rho(\Phi))$ for a large $W$.
\end{enumerate}

\item The linear response is 
\[ \begin{split}
  \delta \rho(\Phi) =\lim_{W\rightarrow \infty} SC - UC^W.
\end{split} \]

\end{enumerate}

\section{Numerical examples}
\label{s:examples}

\subsection{A 21 dimensional solenoid map}
\hfill\vspace{0.1in}

This subsection illustrates the fast adjoint response algorithm on 
a 21-dimensional solenoid map with 20 unstable dimensions.
Here $\cM = \R \times \T^{20}$, and the governing equation is
\[ \begin{split}
  x^1_{n+1} &= 0.05x^1_{n} + \gamma_1 + 0.1 \sum_{i = 2}^{21} \cos(5x^i_{n}) \\
  x^i_{n+1} &= 2x^i_{n} + \gamma_2 (1+x^1_{n}) \sin(2x^i_{n}) \mod 2\pi, \quad \textnormal{for} \quad 2\le i\le 21
\end{split} \]
where the superscript labels the coordinates,
and the instantaneous objective function is 
\[ \begin{split}
  \Phi(x) := (x^1)^3 + 0.005 \sum_{i = 2}^{21} (x^i - \pi)^2.
\end{split} \]
This is almost the same example as in the paper for fast response \cite{fr},
except for that now we have two parameters, $\gamma_1$ and $\gamma_2$.

The default setting, $N=20$ steps in each segment, 
$A=200$ segments, and $W=10$, is used unless otherwise noted.
The code is at \url{https://github.com/niangxiu/far}.

We first let $\gamma = \gamma_1 = \gamma_2$ 
and consider derivatives computed with respect to one parameter.
We choose the default value $\gamma=0.1$.
Figure~\ref{f:change A} shows the convergence of the fast adjoint response algorithm as the orbit gets longer; in particular, the variance of the computed derivative is proportional to $A^{-0.5}$, like the sampling error of an orbit.
Figure~\ref{f:change W} shows that the bias in the average derivative decreases as $W$ increases,
but the variance increases roughly like $W^{0.5}$.
Figure~\ref{f:change parameter} shows that 
the derivative computed by the fast adjoint response correctly reflects 
the trend of the objective as $\gamma$ changes.
The plots we obtained via the fast adjoint response is almost identical 
to that of the fast response \cite{fr};
this is expected, since they are two algorithms for the same quantity.

\begin{figure}[ht] \centering
  \includegraphics[width=0.45\textwidth]{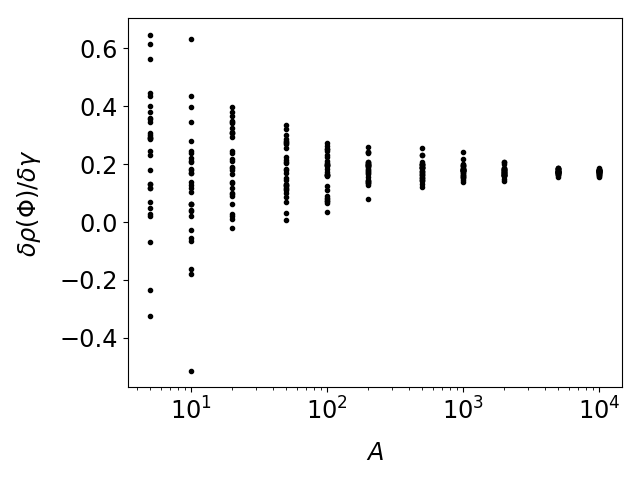}
  \includegraphics[width=0.45\textwidth]{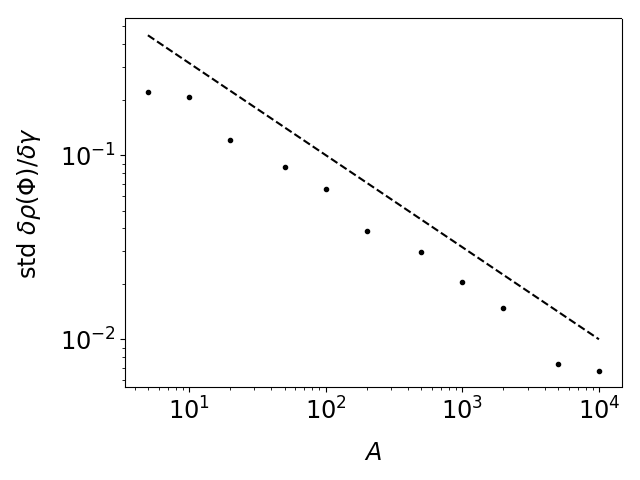}
  \caption{Effects of $A$. Left: derivatives from 30 independent computations for each $A$.
  Right: the sample standard deviation of the computed derivatives,
  where the dashed line is $A^{-0.5}$.}
  \label{f:change A}
\end{figure}

\begin{figure}[ht] \centering
  \includegraphics[width=0.45\textwidth]{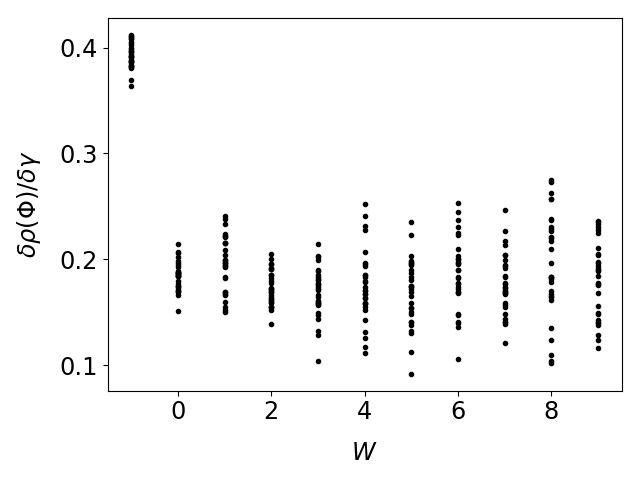}
  \includegraphics[width=0.45\textwidth]{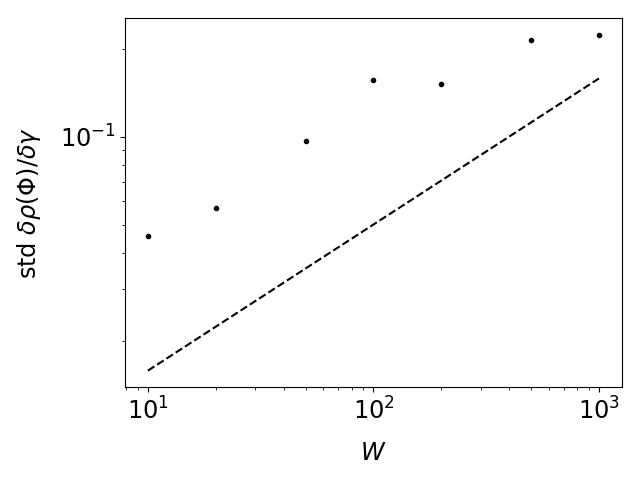}
  \caption{Effects of $W$. Left: derivatives computed by different $W$'s.
  Right: standard deviation of derivatives, where the dashed line is $0.005W^{0.5}$.}
  \label{f:change W}
\end{figure}

\begin{figure}[ht] \centering
  \includegraphics[width=0.5\textwidth]{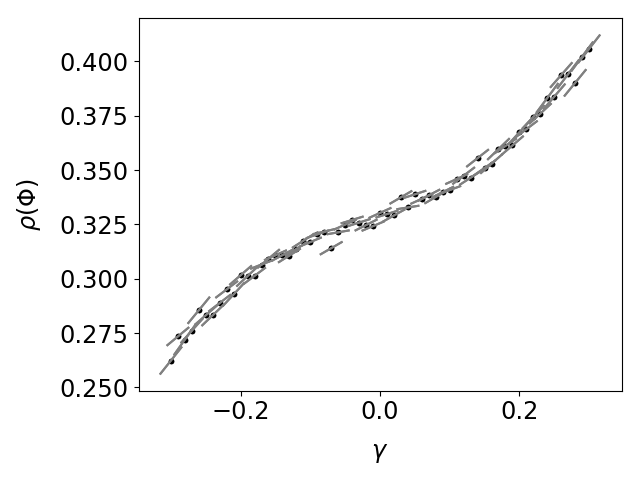}
  \caption{Average objectives for different parameter $\gamma$.
  The grey lines are the derivatives computed by fast adjoint response.
  Both the objective and derivative are computed from the same orbit of $200$ segments.}
  \label{f:change parameter}
\end{figure}

On a single-core 3.0GHz CPU, for $10^4$ segments, 
which is a total of $2\times10^5$ steps,
the time for computing the orbit is 5.6 seconds;
the fast adjoint response on the same orbit takes another 186 seconds.
So the time cost of fast adjoint response is about 34 times of simulating the orbit, or 6 times of the fast response.
Also, our algorithm can run even faster if we use the sparsity of the system, either via sparse matrices or graph tracing.

By the same analysis we have done in the tangent paper,
on our example, fast adjoint response is about $10^{6}$ times faster than ensemble or stochastic adjoint algorithms;
the isotropic finite-element operator method would take a tremendous amount of storage.
Other algorithms do not yet have adjoint versions.
When there is only one parameter and one objective, fast response has cost similar to finite differences,
which require data from several orbits to cancel the noise.
This cost comparison again recovers the simple case of computing derivatives for stable dynamical systems, where a conventional adjoint solver's cost is similar to finite differences.
Again, this hints that the efficiency of fast adjoint response is perhaps close to the best possible precise adjoint algorithms for chaos.

\begin{figure}[ht]
  \centering
  \includegraphics[width=0.6\textwidth] {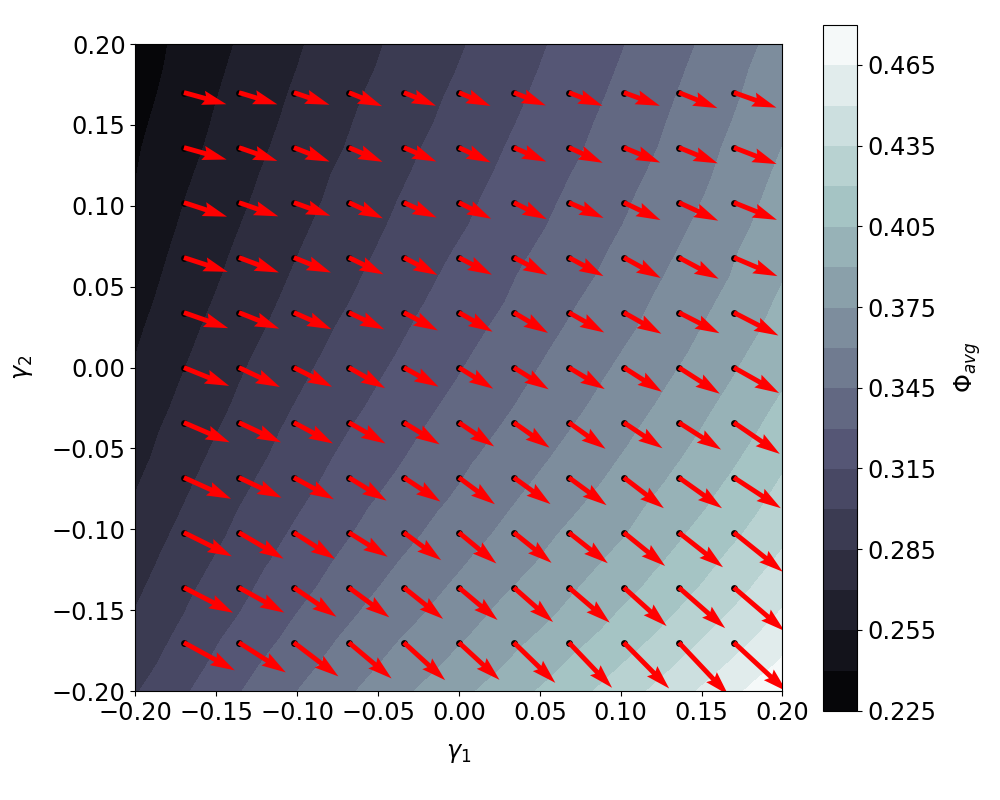}
  \caption{Gradients computed by fast adjoint response 
    and the contour of $\rho(\Phi)$.
    Here $\rho(\Phi)$ is averaged over 30 orbits with $1000$ segments,
    while the gradient is averaged over 30 orbits of $200$ segments.
    The arrow's length is $1/15$ of the gradient.}
  \label{f:contour}
\end{figure}

Finally, we demonstrate the fast adjoint response on two parameters $\gamma_1\neq \gamma_2$.
Fast adjoint response computes sensitivities with respect to multiple parameters with almost no additional cost; 
in fact, the cost of the unstable contribution is also independent of the number of objectives.
Figure~\ref{f:contour} illustrates the contour of $\rho(\Phi)$ and the gradients computed.
Since we use the same length unit for both parameters, gradients should be perpendicular to the level sets of the objective:
our algorithm indeed gives the correct gradient.

\subsection{Non-autonomous systems with stochastic noise}
\hfill\vspace{0.1in}

Our algorithm works on systems with a time-dependent outside force, or time-varying random noise.
In this case, we should modify our definition for stable/unstable subspaces and physical measures, which now varies with time.
Still, if we start a smooth density from the infinite past and evolve to the current step, we always get the physical measure at this step, regardless of the initial density.
We can rerun the derivations of the linear response and the fast response algorithm, an we still has the same expression.

For simplicity, we keep using the map in the previous subsection but reduce the space to $\cM = \R\times \T^{2}$.
The governing equation is 
\[ \begin{split}
  x^1_{n+1} &= 0.05x^1_{n} + \gamma_1 + 0.1 \sum_{i = 2}^{3} \cos(5x^i_{n}) + U \\
  x^i_{n+1} &= 2x^i_{n} + \gamma_2 (1+x^1_{n}) \sin(2x^i_{n}) + U \mod 2\pi, \quad \textnormal{for} \quad 2\le i\le 3,
\end{split} \]
where $U$ is a number generated by the uniform distribution on $[-5,5]$; each appearance of $U$ is generated independently.
We should control the noise level so that it does not ruin the uniform hyperbolicity; the range of $\gamma$ should also be accordingly reduced.
Change the objective function to 
\[ \begin{split}
  \Phi(x) := (x^1)^3 + \sum_{i = 2}^{3} (x^i - \pi)^2,
\end{split} \]
so that the shadowing contribution is on the same order as the unstable contribution, and both contributions can be tested.

\begin{figure}[ht] \centering
  \includegraphics[width=0.45\textwidth]{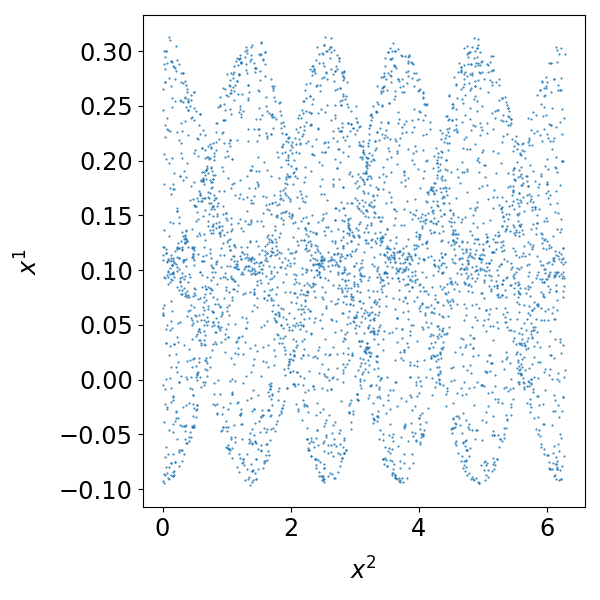}
  \includegraphics[width=0.45\textwidth]{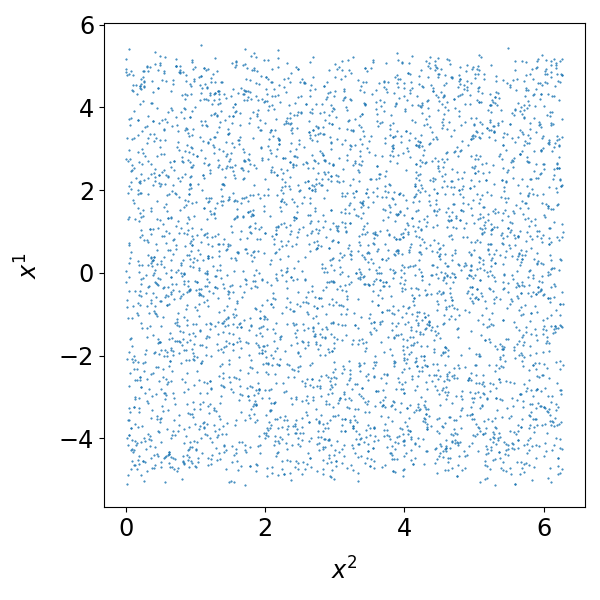}
  \caption{The physical measure approximated by an orbit of $200$ segments. Left: no noise. Right: with noise $U$.}
  \label{f:perk}
\end{figure}

\begin{figure}[ht] \centering
  \includegraphics[height=6cm]{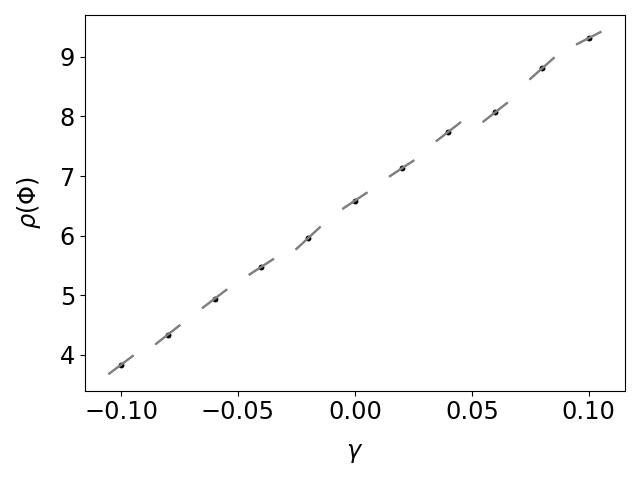}
  \caption{Derivatives for the system with random noise computed with $10^3$ segments. The objectives are averaged over $10^4$ segments.}
  \label{f:simons}
\end{figure}

The density of the physical measure approximated by an orbit is given in figure~\ref{f:perk}.
As we can see, the random forcing $U$ drastically changes the appearance of the physical measure.
Still, we can use the algorithm to compute a good derivative, as shown in figure~\ref{f:simons}.

There might be a caveat for random forcing, that is, we should assume the particular realization of the noise leads to the correct physical measure for a small range of $\gamma$.
We are not sure if this is true due to our limited knowledge, but we can alter our logic a little to get rid of randomness, that is, we fix the random noise once it is sampled, and then we compute the linear response for this determined forcing.
In the examples here we do \textit{not} fix the noise for different $\gamma$, and we still get good results.

\subsection{Lorenz maps with discontinuous conditional measure}
\hfill\vspace{0.1in}

We also apply our algorithm on the Lorenz map, a one-dimensional map on the interval $[0,1]$, which also represents the torus $\T$.
Its expression is
\[ \begin{split}
  f(x) = \begin{cases}
    Bx^2 + 2x + \gamma \sin(2\pi Tx)/2\pi T
    \quad \textnormal{mod 1,} \quad
    \quad\text{if}\quad 0\le x\le 0.5,
    \\
    B(x-1)^2 + 2-2x -\gamma \sin(2\pi Tx)/2\pi T
    \quad \textnormal{mod 1,} \quad 
    \quad\text{otherwise.}
  \end{cases}
\end{split} \]
Here the fixed objective function is $\Phi(x) = x$.

We first set $B=3$.
The graph of this map and its physical measures, approximated by a histogram of a typical orbit is shown in figure~\ref{f:lorenzsystem}.
Note that the density is not even continuous, even though we have only one unstable dimension.
We fix $N=20$ steps in each segment.
This example goes beyond our theoretical assumptions, so we should expect some systematic error; but our algorithm performs reasonably well, as shown in figure~\ref{f:B3}.

\begin{figure}[ht] \centering
  \includegraphics[height=6cm]{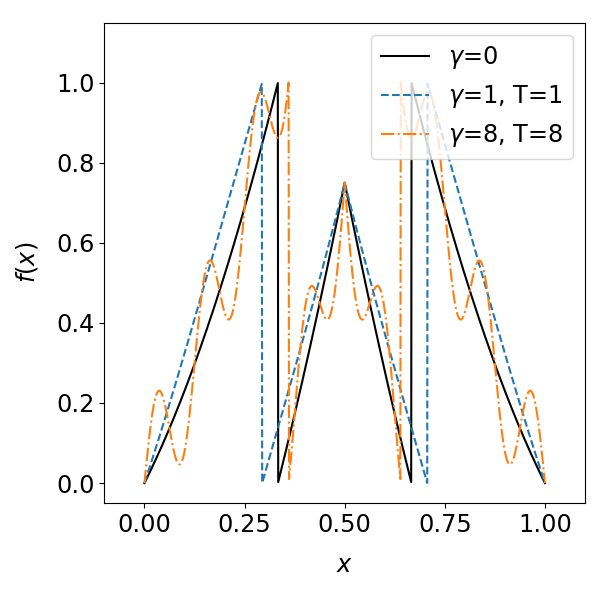}
  \includegraphics[height=6cm]{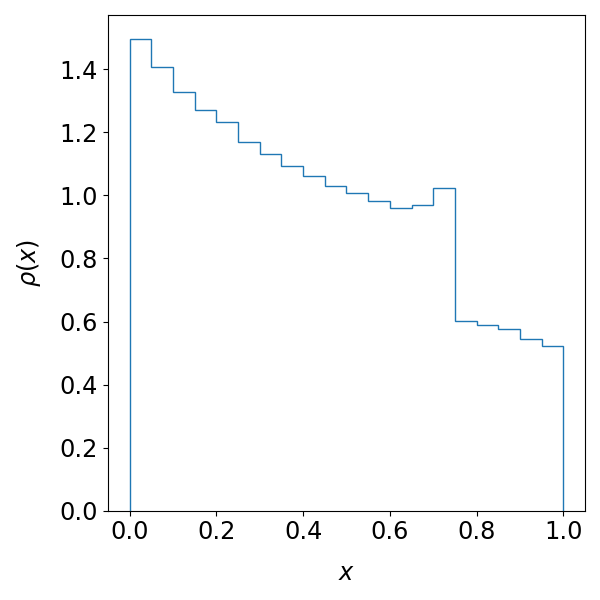}
  \caption{Left: the graph of the Lorenz map for $B=3$. Right: the physical measure approximated by a histogram.}
  \label{f:lorenzsystem}
\end{figure}

\begin{figure}[ht] \centering
  \includegraphics[height=6cm]{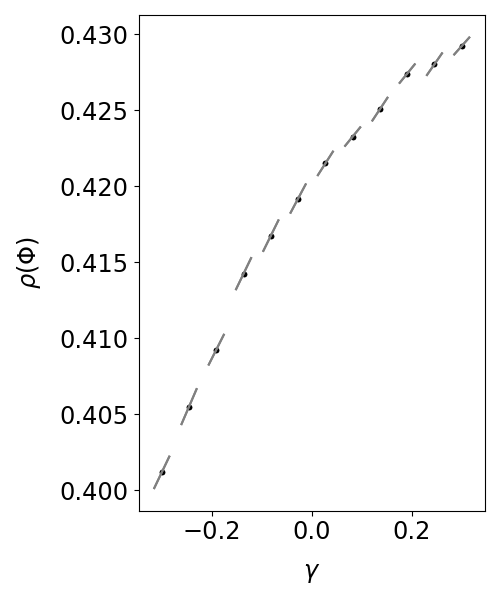}
  \includegraphics[height=6cm]{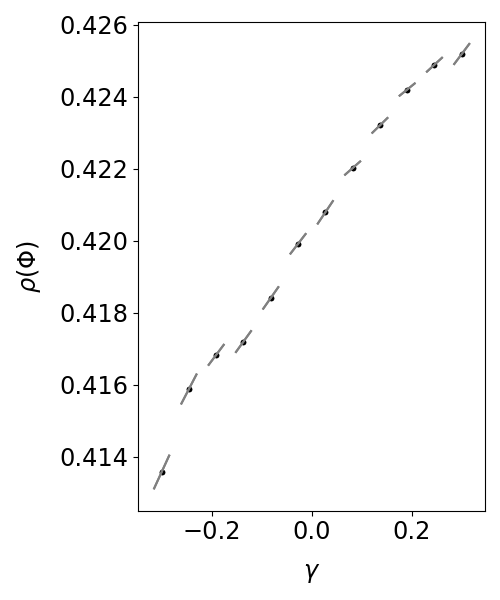}
  \caption{Derivatives for $B=3$ computed with $10^3$ segments. The objectives are averaged over $10^5$ segments. Left to right: $T=1, 8$.}
  \label{f:B3}
\end{figure}

We also use this example to explain a seemingly surprising phenomenon reported in \cite{shadowingNonPhysical}, where a small perturbation caused a large impact on the physical measure. 
Set $B=0$, so that our map is the normal tent map. 
Then we compute the derivatives for $T=1,7,8$, with results shown in figure~\ref{f:B0}; our algorithms works better than $B=3$ since now the conditional measure is smooth. 
As we can see, the derivative for $T=1$ and $8$ are roughly on the same order, but are much larger than $T=7$.
Since the sup norm of the perturbation $X$ decreases with $A$, it seems that we have a large response due to a small $X$.

\begin{figure}[ht] \centering
  \includegraphics[height=6cm]{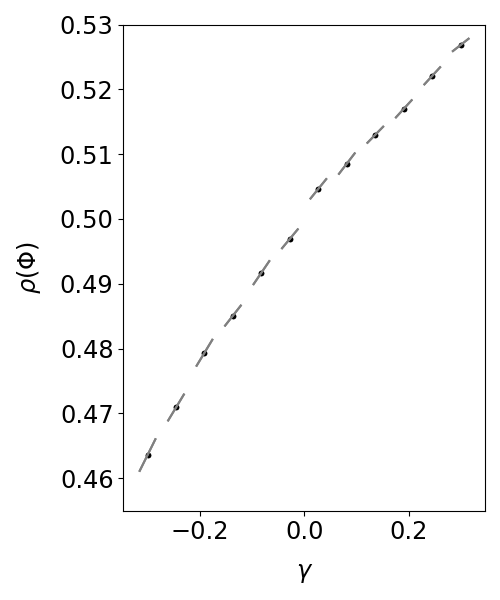}
  \includegraphics[height=6cm]{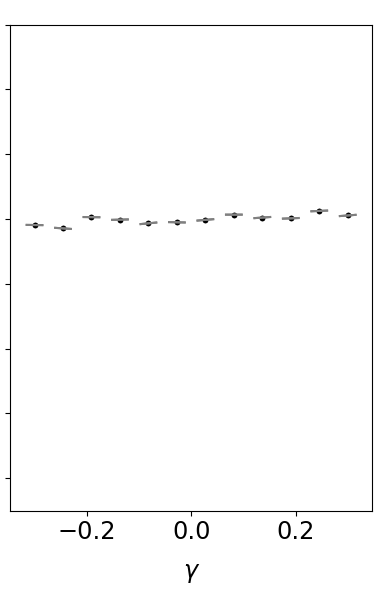}
  \includegraphics[height=6cm]{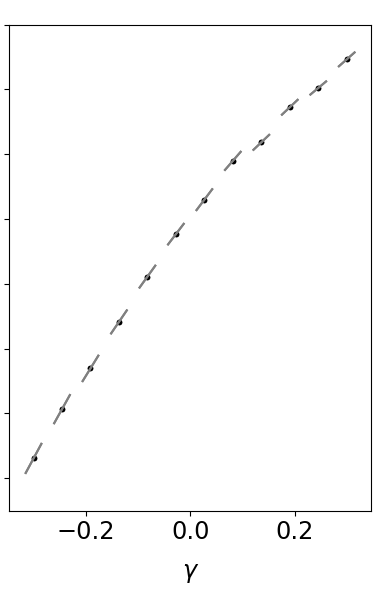}
  \caption{Derivatives for $B$=0 computed with $10^3$ segments. The objectives are averaged over $10^4$ segments. Left to right: $T=1, 7, 8$.}
  \label{f:B0}
\end{figure}

We may thus worry that numerical simulations of chaotic dynamics may fail to predict the true long-term behavior: it could happen, but not for this example.
First, our algorithm accurately computes the linear response.
Second, we know from ergodic theory that for uniform hyperbolic systems, the physical measure is the limit of small random perturbations \cite{young2002srb}.
Third, due to our $v$-divergence formula, we can say more than the previous theory.

The reason for an $X$ with small sup norm to cause a large response is a large $\div^v X$ which resonates with the map $f$.
For $X$ with a small sup norm, by our expressions, the shadowing contribution is small, so is the $\cS \omega X$ part in the unstable contribution.
To have a large response, we need $\div^v X \circ f^n$ to be large for some $n$.
Since $X$ is small, we must have high oscillations to have a large $\div^vX$; moreover, this oscillation needs to be picked up by the map.
For example, the perturbation with $T=7$ does not resonate with the standard tent map, but $T=8$ does.
For typical numerical simulations, the numerical error has no particular reason to resonate with the map, so we should not worry about the validity of simulating a hyperbolic system with enough uniformity, such as this given example.

\section{Discussions}
\hfill\vspace{0.1in}

We list some remarks on the algorithm.
First, the main cost of the algorithm is computing $\omega$ in step~\ref{step:omega} of the procedure.
For a dense 3-tensor $\nabla f_*$, this can be very expensive, since for each $j$, $\omega^j$ requires summing $uM^2$ terms: this is a total of $O(uM^3)$ float operations to get $\omega$.
However, for many engineering applications, such as fluid mechanics and convolutional neural networks, $\nabla f_*$ is sparse.
This means that for each $j$, $\partial^2 f^l/\partial x^k \partial x^j $ is nonzero only for a few $l$ and $k$.
Hence, for these situations, computing $\omega$ costs only $O(uM)$ per step, and the overall complexity is the same order as the nonintrusive shadowing and the fast response algorithm.
Another major cost is $O(u^2M)$ per segment for the occasional QR factorization, which is also needed for the nonintrusive shadowing. 
Hence, the extra cost of the fast response to the nonintrusive shadowing is $O(uM)$ per step.

The cost for a new $X$ is $O(uM)$ per step for a sparse system, since we need to compute $\nu X$, $\tnu X$, and $\div^v X$.
Although in the same order, the constant behind the big-$O$ is smaller than the main part of the algorithm, since we are contracting the 2-tensor $\nabla X$ rather than the 3-tensor $\nabla f_*$.
This is the same as conventional adjoint method, where computing inner-products for a new $X$ is $O(M)$, same order as the cost of solving the adjoint equation.
Moreover, we do not have the $O(u^2M)$ per segment cost for additional $X$.

Although the number of float number operation for computing adjoint solutions 
are similar to tangent ones, adjoint solvers can take longer time.
This is because adjoint solvers runs backward in time,
which requires storing and reading data of the orbit.
It seems that the most time-saving solution is to save checkpoints of the orbit
and then rerun a segment of the orbit when solving the adjoint equations.
This makes the adjoint solver a few times slower than tangent solvers.

Note that the cost for computing the unstable contribution is not only almost independent of the number of parameters; it is also almost independent of the number of objectives.
This is different from conventional adjoint methods, whose cost depends almost linearly on the number of objectives.
On the other hand, we require both tangent and adjoint solvers, whereas conventional adjoint and adjoint shadowing algorithms require only adjoint solvers.

Many practical comments for the fast response algorithm still apply in the adjoint version.
For example, why our work may hold beyond the strong hyperbolicity assumptions we used in the proof;
how to choose the number of steps in each segment;
how to efficiently contract a higher tensor with several lower tensors;
the cost-error estimation and its comparison with ensemble or stochastic methods.

In \cite{Ruesha}, there is a dimension argument to say the unstable contribution is typically small when the unstable dimension is small.
With the new $v$-divergence formula, we can give another estimation of the unstable contribution, which is controlled by the sup norm of $X$ and its derivatives.
In particular, the dimension argument is still useful.
In the unstable contribution, both $\div^v X$ and $\div^v f_*$ are contracted by $e$ and $\eps$, which have $u$ components.
It is still reasonable to expect that the unstable divergence is typically small when $u$ is small.
However, as we mentioned in \cite{Ruesha}, the caveat is that we could have $f$ such that $\div^v f_*$ is large even for $u\ll M$.
Similarly, we can have large $\div^vX$ when $f$ is intentionally designed according to $X$.
This could happen, say, close to the end of an optimization \cite{RepolhoCagliari2021}.

Last but not the least, we report that our current algorithm do \textit{not} converge for logistic maps such as $f(x)=3.8x(1-x)$.
Indeed, the logistic map doesn't have a linear response, but we can use this example to see a case when the algorithm fails, that is, the lack of enough uniformity in the expanding rate explodes the shadowing covector and hence our algorithm.
We do not exactly need uniform hyperbolicity, but we still need some control over the non-uniformity.
It was conjectured that high-dimensional physical systems are uniform enough \cite{gallavotti_chaotic_hypothesis_2006}, but mean-field models can offer counter examples \cite{Wormell2019,wormell22}.
Nevertheless, now we have a good solution for computing linear response of hyperbolic systems in discrete time.
Moreover, it seems that our formulas and algorithms for hyperbolic systems are basic building blocks for sampling non-hyperbolic linear responses by an orbit: this is indeed the case for continuous-time flows \cite{Ni_asl}.

\section*{Acknowledgements}
The author is very grateful to Caroline Wormell for helping on references.
This research is partially supported by the China Postdoctoral Science Foundation 2021TQ0016 and the International Postdoctoral Exchange Fellowship Program YJ20210018.

\section*{Data availability statement}
The code used in this manuscript is at \url{https://github.com/niangxiu/far}.
There is no other associated data.

\appendix
\section{Solving the least-squares version of the adjoint shadowing problem}
\label{a:schur}

The adjoint shadowing problem in this paper is different from \cite{Ni_nilsas}.
The orthogonal projection version in step~\ref{step:ni} of section~\ref{s:procedure} can be solved quite easily and we found no significant performance difference with other versions.
For readers more familiar with the previous least-squares version of this problem, we give the detailed formulas for the tridiagonal matrix algorithm, it is more complicated but might have better performance.

The nonintrusive adjoint shadowing has a least-squares version.
The least-squares idea was first used in the least-squares shadowing method, whose cost averages to somewhat more than $O(M^2)$ flop per step for sparse systems, depending on iterative solvers \cite{wang2014convergence}. 
The multiple shooting shadowing method samples the orbit by checkpoints; its average cost is $O(M^2)$ per step, and the storage management is better \cite{Blonigan_MSS}.
This cost should be similar to the stable part of the original blended response algorithm, whose systematic error is somewhat smaller than shadowing \cite{Abramov2008}.

The nonintrusive idea brings down the cost to $O(uM)$ flop per step plus $O(u^2M)$ per segment.
To keep using the least-squares, first compute and store the covariance matrix and the inner product,
  \[ \begin{split}
    C_\alpha := \sum_{n=1}^{N} \underline \eps_{\alpha,n}^T \underline \eps_{\alpha,n}, \quad
    d_\alpha := \sum_{n=1}^{N} \underline \eps_{\alpha,n}^T \nu'_{\alpha,n}, \quad
    \tilde d_\alpha := \sum_{n=1}^{N} \underline \eps_{\alpha,n}^T \tilde \nu'_{\alpha,n} dt .
  \end{split} \]
Note that computing $C$ as listed above would cost $O(u^2M)$ each step, so we should sample $C$ by a few steps in each segment.

The least-squares version of the nonintrusive adjoint shadowing solves
\[ \begin{split}
  \min \sum_n |\nu_n|^2 ,
  \quad  \mbox{where} \quad
  \nu =  \nu ' + \ueps a \,.
\end{split} \]
Using segments' data, we write the above as
\[\begin{split}
  &\min_{\{a_\alpha\}} 
  \sum_{\alpha=0}^{A-1} \frac 12 a_\alpha^T C_\alpha a_\alpha + d_\alpha^T a_\alpha \\
  \mbox{s.t. }& 
    a_{\al-1} = R^T_\al (a_{\al} + b_\al).
  \quad \alpha=1,\ldots,A-1.
\end{split}\]
The Lagrange function is:
\[ \begin{split}
  \sum_{\alpha=0}^{A-1} \frac 12 a_\alpha^T C_\alpha a_\alpha + d_\alpha^T a_\alpha 
  + \sum_{\alpha=1}^{A-1} \lambda_\al^T \left(a_{\al-1} 
  - R_\al^T (a_{\al} + b_{\al}) \right).
\end{split} \]
where $\lambda_\al$ is the Lagrange multiplier for the continuity condition at step $\al N$.
The minimizer is given by
\[ \begin{split}
  \mat{ C & B\\ B^T & 0 }  
  \mat{ a\\  \lambda } 
  =\mat{ -d \\ b' }\;,
  \quad \textnormal{where} 
\end{split} \]
\[ \begin{split}
  C = \mat{C_0 \\ &C_1 \\&& \ddots \\ &&&C_{A-1}}
  , \quad
  B = \mat{I \\ -R_1 &I \\&\ddots &\ddots \\ &&-R_{A-2} &I \\&&&-R_{A-1} }, \\
\end{split}\]
\[ \begin{split}
  a = \mat{a_0 \\ \vdots\\a_{A-1}}
  , \quad
  d = \mat{d_0 \\ \vdots\\d_{A-1}}
  , \quad
  \lambda = \mat{\lambda_1 \\ \vdots\\ \lambda_{A-1} }
  , \quad
  b' = \mat{R_1^Tb_1 \\ \vdots\\R_{A-1}^Tb_{A-1}}.
\end{split} \]
Here $\{C_\al\}_{\al=0}^{A-1}$, $\{R_\al\}_{\al=1}^{A-1} \subset \R^{u\times u}$;
$\{a_\al\}_{\al=0}^{A-1}$, $\{d_\al\}_{\al=0}^{A-1}$, $\{\lambda_\al\}_{\al=1}^{A-1}$, $\{b_\al\}_{\al=1}^{A-1}\subset \R^u$.
Note that $C_\al$ and $C$ are symmetric matrices.

We first solve the Schur complement for $\lambda$ \cite{Blonigan_MSS},
\[ \begin{split}
  B^T C^{-1} B \lambda = -( B^T C^{-1}d + b') .
\end{split} \]
We can write the left side in block form,
\[ \begin{split}
  C^{-1} B = \mat{C_0^{-1} \\ -D_1 &\ddots\\ &\ddots &C_{A-2}^{-1} \\&&-D_{A-1}}, 
  \Lambda := B^T C^{-1} B = \mat{E_1 & -D_1^T \\ -D_1 &E_2 &\ddots \\&\ddots &\ddots & -D_{A-2}^T \\ 
  &&-D_{A-2} &E_{A-1} }, 
\end{split} \]
where
\[ \begin{split}
  D_\al:= C_\al ^{-1} R_\al, \quad
  E_\al:= C_{\al-1}^{-1} + R_\al^T D_\al,  
  \quad \al=1,\cdots,A-1.
\end{split} \]
For the right side, denote
\[ \begin{split}
  y = \mat{y_1\\\vdots\\y_{A-1}}
  :=  -( B^T C^{-1}d + b'),
\end{split} \]
where
\[ \begin{split}
  y_\al = D_\al^{T}d_\al - C_{\al-1}^{-1}d_{\al-1} - R_\al^T b_\al ,
  \quad \al=1,\cdots,A-1.
\end{split} \]

The linear equation system $\Lambda \lambda= y$ is a block-tridiagonal matrix;
to solve it efficiently, we use the tridiagonal matrix algorithm,
whose cost is only $O(A)$.
First, we run a `forward chasing' to eliminate below-diagonal blocks in $\Lambda$,
\[ \begin{split}
  W:= D_{\al-1} E_{\al-1}^{-1},
  \quad \textnormal{} \quad 
  E_\al \leftarrow E_\al - W D_{\al-1}^T,
  \quad \textnormal{} \quad 
  y_\al \leftarrow y_\al + W y_{\al-1},\\
  \quad \textnormal{sequentially for} \quad 
  \al = 2,3,\cdots,A-1.
\end{split} \]
Here `$\leftarrow$' means to replace with new values.
Then we run a `backward chasing' to compute $\lambda_\al$.
\[ \begin{split}
&\lambda_{A-1} = E_{A-1}^{-1} y_{A-1};\\
&\lambda_{\al} = E_{\al}^{-1}(D_\al^T \lambda_{\al+1} +  y_{\al}),
  \quad \textnormal{sequentially for} \quad 
  \al = A-2,\cdots,1.
\end{split} \]

Then we solve for $a$ via
\[ \begin{split}
  a = - C^{-1}(B \lambda +d) .
\end{split} \]
More specifically,
\[ \begin{split}
  &a_0 = C_0^{-1}(-d_0-\lambda_1);\\
  &a_\al = C_\al^{-1}(-d_\al-\lambda_{\al+1}+R_\al \lambda_\al),\quad \al=2,\cdots A-2;\\
  &a_{A-1} = C_{A-1}^{-1}(-d_{A-1}+R_{A-1} \lambda_{A-1}).\\
\end{split} \]

The least-squares version of nonintrusive shadowing allows $u'>u$, where $u'$ is the actual number of homogeneous tangent solutions we compute.
But for the other part of the fast response algorithm, we need $u'=u$ anyway.

\bibliographystyle{abbrv}
{\footnotesize\bibliography{MyCollection}}

\end{document}